\documentclass [twoside,reqno,  12pt] {amsart}

\usepackage{psfrag}

\usepackage{hyperref}

\usepackage{amsfonts}
\usepackage{amssymb}
\usepackage{a4}
\usepackage{color}

\usepackage{soul}

\newtheorem{thm}{Theorem}[section]
\newtheorem{cor}[thm]{Corollary}
\newtheorem{lem}[thm]{Lemma}
\newtheorem{prop}[thm]{Proposition}

\newtheorem{defn}[thm]{Definition}
\newtheorem{rem}[thm]{Remark}

\theoremstyle{definition}

\numberwithin{equation}{section}

\renewcommand{\Re}{\hbox{Re}\,}
\renewcommand{\Im}{\hbox{Im}\,}

\newcommand{\R}{\mathbb{R}}

\newcommand{\supp}{\operatorname{supp}}

\def\hat{\widehat}
\def\tilde{\widetilde}
\def \bfo {\begin {eqnarray*} }
\def \efo {\end {eqnarray*} }
\def \ba {\begin {eqnarray*} }
\def \ea {\end {eqnarray*} }
\def \beq {\begin {eqnarray}}
\def \eeq {\end {eqnarray}}
\def \supp {\hbox{supp }}

\def \p {\partial}

\def\hat{\widehat}
\def\tilde{\widetilde}
\def \bfo {\begin {eqnarray*} }
\def \efo {\end {eqnarray*} }
\def \ba {\begin {eqnarray*} }
\def \ea {\end {eqnarray*} }
\def \beq {\begin {eqnarray}}
\def \eeq {\end {eqnarray}}
\def \supp {\hbox{supp }}

\def \p {\partial}

\begin{document}

 \title[Reconstructing a potential perturbation]{Reconstructing a potential perturbation of the biharmonic operator on transversally anisotropic manifolds}

\author[Yan]{Lili Yan}

\address
        {Lili Yan, Department of Mathematics\\
         University of California, Irvine\\ 
         CA 92697-3875, USA }

\email{liliy6@uci.edu}

\maketitle

\begin{abstract} 
We prove that a continuous potential $q$ can be constructively determined from the knowledge of the Dirichlet--to--Neumann map for the perturbed biharmonic operator $\Delta_g^2+q$ on a conformally transversally anisotropic Riemannian manifold of dimension $\ge 3$ with boundary, assuming that the geodesic ray transform on the transversal manifold is constructively invertible. This is a constructive counterpart of the uniqueness result of \cite{Yan_2020}. In particular, our result is applicable and new in the case of  smooth bounded domains in the $3$--dimensional  Euclidean space  as well as in the case of $3$--dimensional admissible manifolds.  

\end{abstract}

\section{Introduction and statement of results}

Let  $(M,g)$ be a smooth compact oriented Riemannian manifold of dimension $n\ge 3$ with smooth boundary $\p M$. Let $\gamma$ be the Dirichlet trace operator defined by 
\begin{equation}
\label{int_trace}
\gamma: H^2(M^{\text{int}})\to H^{3/2}(\p M)\times H^{1/2}(\p M), \quad \gamma u=(u|_{\p M}, \p_\nu u|_{\p M}),
\end{equation}
which is bounded and surjective, see \cite[Theorem 9.5]{Grubb_book}. Here and in what follows $M^{\text{int}}=M\setminus \p M$,  $H^s(M^{\text{int}})$ and $H^{s}(\p M)$, $s\in \R$, are the standard $L^2$--based Sobolev spaces on $M^{\text{int}}$ and its boundary $\p M$, respectively, and $\nu$ is the exterior unit normal to $\p M$. We also let $H^2_0(M^{\text{int}})=\{u\in H^2(M^{\text{int}}): \gamma u=0\}$. Let $-\Delta_g=-\Delta$ be the Laplace--Beltrami operator on $M$, and let $\Delta^2$ be the biharmonic operator on $M$. Let $q\in C(M)$. By standard arguments, see for instance \cite[Appendix A]{Krup_Uhlmann_2016}, the operator
\begin{equation}
\label{int_1}
\Delta^2+q: H^2_0(M^{\text{int}})\to H^{-2}(M^{\text{int}})=(H^2_0(M^{\text{int}}))',
\end{equation}
is Fredholm of index zero and has a discrete spectrum. We shall assume throughout the paper that 
\begin{itemize}
\item[(A)] $0$ is not in the spectrum of the operator \eqref{int_1}.
\end{itemize}
Thus, for any $f=(f_0,f_1)\in H^{3/2}(\p M)\times H^{1/2}(\p M)$, the Dirichlet problem
\begin{equation}
\label{int_2}
\begin{cases}
(\Delta^2+q)u=0 & \text{in}\quad M^{\text{int}},\\
\gamma u=f & \text{on}\quad \p M,
\end{cases}
\end{equation}
has a unique solution $u\in H^2(M^{\text{int}})$, depending continuously on $f$. Physically, the Dirichlet boundary condition in \eqref{int_2} corresponds to the clamped plate equation, see  \cite{Gazzola_2010}.
 We define the Dirichlet--to--Neumann map  $\Lambda_q$ by 
\begin{equation}
\label{int_2_DN_map}
\langle \Lambda_q f,g \rangle_{H^{-3/2}(\p M)\times H^{-1/2}(\p M), H^{3/2}(\p M)\times H^{1/2}(\p M)}=\int_M(\Delta u)(\Delta v)dV+\int_M quvdV, 
\end{equation}
where $g=(g_0,g_1)\in H^{3/2}(\p M)\times H^{1/2}(\p M)$, $v\in H^2(M^{\text{int}})$ is such that $\gamma v=g$, and $u$ is the solution to \eqref{int_2}.  The linear map $\Lambda_q$ is well defined and 
\[
\Lambda_q: H^{3/2}(\p M)\times H^{1/2}(\p M)\to H^{-3/2}(\p M)\times H^{-1/2}(\p M)
\]
is continuous, see \cite[Appendix A]{Krup_Uhlmann_2016}.  This corresponds to the fact that in the weak sense we have $\Lambda_q f=(-\p_\nu(\Delta u)|_{\p M}, \Delta u|_{\p M})$.  

Note that working with solutions $u\in H^4(M^{\text{int}})$ of the equation $ (\Delta^2+q)u=0$, the explicit description for the Laplacian in the boundary normal coordinates, see \eqref{eq_2_1} below, together with boundary elliptic regularity, see \cite[Theorem 11.14]{Grubb_book}, shows that the knowledge of the graph of the Dirichlet--to--Neumann map  $\Lambda_q$, $\{(f, \Lambda_q f): f\in H^{\frac{7}{2}}(\p M)\times H^{\frac{5}{2}}(\p M)\}$ is equivalent to the knowledge of the set of the Cauchy data,
\[
\{(u|_{\p M}, \p_\nu u|_{\p M}, \p_\nu^2 u|_{\p M}, \p_\nu^3 u|_{\p M}): \, u\in H^4(M^{\text{int}}), \, (\Delta^2+q)u=0\text{ in } M^{\text{int}}\}.
\]

The inverse boundary problem for a potential perturbation of the biharmonic operator is to determine the potential $q$ in $M$ from the knowledge of the Dirichlet--to--Neumann map $\Lambda_q$.  In the case of domains in the Euclidean space $\R^n$ with $n\ge 3$, this problem was solved in \cite{Ikehata_1991}, \cite{Isakov_1991} showing that the bounded potential $q$ can indeed be recovered from the knowledge of the Dirichlet--to--Neumann map $\Lambda_q$, see \cite{Krup_Uhlmann_2016} for the case of unbounded potentials.  We refer to \cite{Krup_Lassas_Uhlmann_2014}, \cite{Krup_Lassas_Uhlmann_2012}
where the inverse boundary problem of determination of a first order perturbation of the biharmonic operator was studied in the Euclidean case, see also \cite{Brown_Gauthier_2021}, \cite{Assylbekov_2016}, \cite{Assylbekov_2016_corrigendum},  \cite{Assylbekov_Iyer_2019} for the case of non-smooth perturbations, and  \cite{Bhattacharyya_Ghosh_2020}, \cite{Ghosh_Krishnan_2016} for the case of second order perturbations.

Going beyond the Euclidean setting, the global uniqueness in the inverse boundary problem for zero and first order perturbations of the biharmonic operator was only obtained in the case when the manifold $(M,g)$ is admissible in  \cite{Assylbekov_Yang_2017}, see Definition \ref{def_admissible} below, and in the more general case when $(M,g)$ is CTA  (conformally transversally anisotropic, see Definitions \ref{def_CTA}) with the injective geodesic X-ray transform on the transversal manifold $(M_0, g_0)$ in \cite{Yan_2020}. The works  \cite{Assylbekov_Yang_2017}  and \cite{Yan_2020} are extensions of the fundamental works \cite{DKSaloU_2009} and \cite{DKurylevLS_2016} which initiated this study in the case of perturbations of the Laplacian. 

\begin{defn}
\label{def_CTA}
A compact Riemannian manifold $(M,g)$ of dimension $n\ge 3$ with boundary $\p M$ is called conformally transversally anisotropic (CTA) if $M\subset\subset \R\times M_0^{\text{int}}$  where $g= c(e \oplus g_0)$, $(\R,e)$ is the Euclidean real line, $(M_0,g_0)$ is a smooth compact $(n-1)$--dimensional manifold with smooth boundary, called the transversal manifold, and $c\in C^\infty(M)$ is a positive function. 
\end{defn}

\begin{defn}
\label{def_admissible}
A compact Riemannian manifold $(M,g)$ of dimension $n\ge 3$ with boundary $\p M$ is called admissible if it is CTA and the transversal manifold $(M_0,g_0)$ is simple, meaning that 
for any $p\in M_0$, the exponential map $\exp_p$ with its maximal domain of definition in $T_p M_0$ is a diffeomorphism onto $M_0$, and $\p M_0$ is strictly convex.

\end{defn}

The proofs of the global uniqueness results in the works \cite{DKSaloU_2009}, \cite{DKurylevLS_2016}  \cite{Assylbekov_Yang_2017}, \cite{Yan_2020} rely on construction of complex geometric optics solutions based on the techniques of Carleman estimates with limiting Carleman weights. Thanks to the work \cite{DKSaloU_2009}, we know that the property of being a CTA manifold guarantees the existence of limiting Carleman weights.

Once uniqueness results for inverse boundary problems have been established, one is interested in upgrading them to a  reconstruction procedure. The reconstruction of a potential perturbation of the Laplacian from boundary measurements in the Euclidian space was obtained in the pioneering works \cite{Nachman_1988} and \cite{Novikov_1988}, see also \cite{Henkin_Novikov_1987}. We refer to \cite{Nachman_Street_2010} for reconstruction in the case of partial data inverse boundary problems. In the case of admissible manifolds, a reconstruction procedure for a potential perturbation of the Laplacian was given in \cite{Kenig_Salo_Uhlmann_2011}, complementing the uniqueness result of \cite{DKSaloU_2009}, see also \cite{Assylbekov_2017}.  In the case of more general CTA manifolds whose transversal manifolds enjoy the constructive invertibility of the geodesic ray transform, a reconstruction procedure for a potential perturbation of the Laplacian was established in \cite{Feizm_K_Oksan_Uhl}, complementing the uniqueness result of \cite{DKurylevLS_2016}.
We refer to  \cite{Belishev_2011}, \cite{Belishev_2015} for the reconstruction of a Riemannian manifold from the dynamical data.

Turning the attention to inverse boundary problems for a potential perturbation of the biharmonic operator, to the best of our knowledge, there is no reconstruction procedure available in the literature and the purpose of this paper is to provide such a reconstruction procedure. Our result will be stated in the most general setting possible, i.e. on a CTA manifold whose transversal manifold enjoys the constructive invertibility of the geodesic ray transform, but it is applicable and new already in the case of  smooth bounded domains in the $3$--dimensional Euclidean space and in the case of $3$--dimensional admissible manifolds. To state our result, we shall need the following definition. 
\begin{defn}
We say that the geodesic ray transform on the transversal manifold $(M_0,g_0)$ is constructively invertible  if any function $f\in C(M_0)$ can be reconstructed from the knowledge of its integrals over all non-tangential geodesics in $M_0$. Here a unit speed geodesic $\gamma:[0,L]\to M_0$ is called non-tangential if $\dot{\gamma}(0), \dot{\gamma}(L)$ are non-tangential vectors on $\p M_0$ and $\gamma(t)\in M_0^{\emph{\text{int}}}$ for all $0<t<L$.  
\end{defn}

Our main result is as follows, and it gives a constructive counterpart of the uniqueness result of \cite{Yan_2020}. 
\begin{thm}
\label{thm_main}
Let $(M,g)$ be a given CTA manifold and assume that the geodesic ray transform on the transversal manifold $(M_0,g_0)$ is constructively invertible. Let $q\in C(M)$ be such that assumption (A) is satisfied. Then the knowledge of $\Lambda_{q}$ determines $q$ in $M$  constructively. 
\end{thm}

Combining Theorem \ref{thm_main} with the constructive invertibility of the geodesic ray transform on a simple two-dimensional Riemannian manifold, see \cite{Pestov_Uhlmann_2004}, \cite{Krishnan_2010}, \cite{Salo_Uhlmann_2011}, see also \cite{Monard_2014}, \cite{Monard_2014_2},  we obtain the following unconditional result. 
\begin{cor}
\label{cor_main}
Let $(M,g)$ be a given $3$--dimensional admissible manifold, and let $q\in C(M)$ be such that assumption (A) is satisfied. Then the knowledge of $\Lambda_{q}$ determines $q$ in $M$  constructively. 
\end{cor}

\begin{rem}
As explained in \cite{DKSaloU_2009},  bounded smooth domains in the Euclidean space are examples of admissible manifolds, and therefore, Corollary \ref{cor_main} is applicable and new in this case. 

\end{rem}

\begin{rem} Beyond the case of a simple two-dimensional Riemannian manifold,  the constructive invertibility of the geodesic ray transform is also known in particular in the following situations:
\begin{itemize}
\item  $(M_0,g_0)$ is a two-dimensional Riemannian manifold with strictly convex boundary, no conjugate points, and the hyperbolic trapped set (these conditions are satisfied in negative curvature, in particular), see  \cite{Guillarmou_Monard_2017}.

\item $(M_0, g_0)$ is of dimension $n\ge 3$, has a strictly convex boundary and is globally foliated by strictly convex hypersurfaces, see \cite{Uhlmann_Vasy_2016}. 

\end{itemize}

\end{rem}

\begin{rem}
The work \cite{Yan_2020} establishes that not only a continuous potential but an entire continuous first order perturbation can be determined uniquely from the knowledge of the set of the Cauchy data on the boundary of a CTA manifold provided that the geodesic ray transform on the transversal manifold is injective, and therefore, it would be interesting to propose a reconstruction procedure of the recovery of a full first order perturbation. We shall address this question in a future work. To the best of our knowledge, there are no reconstruction results even in the case of a first order perturbation of the Laplacian on admissible manifolds and the only available result is the work \cite{Campos_2019} in the case of compact domains contained in cylindrical manifolds of the form $\R\times\mathbb{T}^d$ with $\mathbb{T}^d$ being the $d$-dimensional torus, $d\ge 2$, see also \cite{Salo_2006} for the Euclidean case. Note that the problem of determining a first order perturbation of the biharmonic operator 
appears to be more challenging, as here one has to recover a first order perturbation uniquely while in the case of the Laplacian, one only needs to determine it up to a gauge transformation, which is only the first step in the corresponding program for the biharmonic operator, see \cite{Yan_2020}. 
\end{rem}

Let us proceed to discuss the main ideas in the proof of Theorem \ref{thm_main}.  The first step is the derivation of the integral identity,
\begin{equation}
\label{eq_int_integral}
\int_M q u_1\overline{u_2}dV=\langle (\Lambda_q-\Lambda_0)\gamma u_1, \gamma \overline{u_2} \rangle_{H^{1/2}(\p M)\times H^{3/2}(\p M), H^{-1/2}(\p M)\times H^{-3/2}(\p M)},
\end{equation}
where $u_1, u_2\in L^2(M)$ are solutions to $(\Delta^2+q)u_1=0$ and $\Delta^2u_2=0$ in $M^{\text{int}}$. The next step is to test the integral identity \eqref{eq_int_integral} agains suitable complex geometric optics solutions $u_1$ and $u_2$. Working on a general CTA manifold, we shall obtain such solutions based on Gaussian beam quasimodes for the conjugated biharmonic operator, constructed on $M$ and localized to non-tangential geodesics on the transversal manifold $M_0$ times $\R_{x_1}$. Such solutions were constructed in \cite{Yan_2020} without any notion of uniqueness involved. In this paper, we propose an alternative construction to produce complex geometric optics solutions enjoying a uniqueness property. The key step in the proof is the constructive determination of the Dirichlet trace $\gamma u_1$ on $\p M$ of the unique complex geometric optics solution $u_1$ from the knowledge of the Dirichlet--to--Neumann map $\Lambda_q$.  Once this step is carried out, the quantity on the right hand side of \eqref{eq_int_integral} is reconstructed thanks to the knowledge of the manifold $M$ and $\Lambda_q$. Another ingredient in the proof is the boundary reconstruction formula for $q|_{\p M}$ from the knowledge of $\Lambda_q$. Using it together with the constructive invertibility of the geodesic ray transform and following the standard argument, see \cite{DKurylevLS_2016}, \cite{Feizm_K_Oksan_Uhl}, we reconstruct the potential $q$ from the left hand side of \eqref{eq_int_integral}, with $u_1$ and $u_2$ being the complex geometric optics solutions.

To the best of our knowledge there are two approaches to the reconstruction of the Dirichlet boundary traces of suitable complex geometric optics solutions to the Schr\"odinger equation in the Euclidean space in the literature. In the first one,  suitable complex geometric optics solutions are constructed globally on all of $\R^n$, enjoying uniqueness properties characterized by decay at infinity, see \cite{Nachman_1988}, \cite{Novikov_1988}, while in the second one, complex geometric optics solutions are constructed by means of Carleman estimates on a bounded domain, and the notion of uniqueness is obtained by restricting the attention to solutions of minimal norm, see \cite{Nachman_Street_2010}. In both approaches, the boundary traces of the complex geometric optics solutions in question are determined as unique solutions of well posed integral equations on the boundary of the domain, involving the Dirichlet--to--Neumann map along with other known quantities.  In the proof of Theorem \ref{thm_main} in order to reconstruct the Dirichlet trace $\gamma u_1$ on $\p M$ of the unique complex geometric optics solution $u_1$ from the knowledge of the Dirichlet--to--Neumann map $\Lambda_q$, we follow the second approach, adapting the simplified version of it given in  \cite{Feizm_K_Oksan_Uhl} to the case of perturbed biharmonic operators.  

Finally, let us mention that similarly to the reconstructions results of \cite{Kenig_Salo_Uhlmann_2011} and  \cite{Feizm_K_Oksan_Uhl}, we make no claims regarding practicality of the reconstruction procedure developed in this paper.  Our purpose merely is to show that all the steps in the proof of the uniqueness result of  \cite{Yan_2020} can be carried out constructively.

This article is organized as follows. In Section \ref{section_2} we collect some essentially well known results related to the maximal domain of the biharmonic operator and boundary traces needed in the proof of Theorem \ref{thm_main}. The derivation of the integral identify \eqref{eq_int_integral} is also given in Section \ref{section_2}. In Section \ref{sec_Nachman-Street}  we present an extension of the Nachman--Street method \cite{Nachman_Street_2010} for the constructive determination of the boundary traces of suitable complex geometric optics solutions, developed for the Schr\"odinger equation, to the case of the perturbed biharmonic equation.  In Section \ref{section_4}, we give a construction of complex geometric optics solutions to the perturbed biharmonic equations enjoying uniqueness property and complete the proof of Theorem \ref{thm_main}. Finally, a reconstruction formula for the boundary traces of a continuous potential from the knowledge of $\Lambda_q$ for the perturbed biharmonic operator is established in Appendix \ref{appendix}.

\section{The Hilbert space $H_{\Delta^2}(M)$ and boundary traces}
\label{section_2}

The purpose of this section is to collect some essentially well known results needed in the proof of Theorem \ref{thm_main}, see also \cite{Grubb_book}, \cite{Lions_Magenes_1972}. Since we are dealing with the biharmonic operator $
\Delta^2$ rather than the Laplacian, some of the proofs are provided for the convenience of the reader. 

Let $(M,g)$ be a smooth compact oriented Riemannian manifold of dimension $n\ge 3$ with smooth boundary $\p M$. 
We shall need the following Green formula for $\Delta^2$, valid for $u,v\in H^4(M^{\text{int}})$, 
\begin{equation}
\label{eq_Green_1}
\begin{aligned}
\int_M (\Delta^2 u) vdV-\int_M u(\Delta^2 v)dV=&\int_{\p M} \p_\nu u (\Delta v)dS -\int_{\p M} u\p_\nu(\Delta v)dS\\
&+\int_{\p M} \p_\nu(\Delta u) vdS-\int_{\p \Omega} (\Delta u)\p_\nu vdS,
\end{aligned}
\end{equation}
where $\nu$ is the unit exterior normal vector to $\p M$, $dV$ and $dS$ are the Riemannian volume elements on $M$ and $\p M$, respectively, see \cite{Grubb_book}.

We shall also need the following expressions for the operators $\Delta$ and $\p_\nu \Delta$ on the boundary of $M$, valid for $v\in H^4(M^{\text{int}})$,
\begin{equation}
\label{eq_2_1}
\begin{aligned}
&\Delta v= \p_\nu^2 v+H\p_\nu v+\Delta_t v\quad \text{on}\quad \p M,\\
&\p_\nu \Delta v= \p_\nu^3 v+\p_\nu H\p_\nu v+H\p_\nu^2 v+\Delta_t \p_\nu v\quad \text{on}\quad \p M,\
\end{aligned}
\end{equation}
where $H\in C^\infty(M)$ and $\Delta_t=\Delta_{g|_{\p M}}$ is the tangential Laplacian on $\p M$, see 
\cite{Lee_Uhlmann_1989}.

Consider the Hilbert space 
\[
H_{\Delta^2}(M)=\{u\in L^2(M): \Delta^2 u\in L^2(M)\},
\]
equipped with the norm
\[
\|u\|^2_{H_{\Delta^2}(M)}=\|u\|^2_{L^2(M)}+\|\Delta^2 u\|_{L^2(M)}^2.
\]
The space $H_{\Delta^2}(M)$ is the maximal domain of the bi-Laplacian $\Delta^2$, acting on $L^2(M)$.

We shall  need the following result concerning  the existence of traces of functions in $H_{\Delta^2}(M)$. 
\begin{lem}
\label{lem_trace_map}
\begin{itemize} 
\item[(i)] The trace map $\gamma_j: C^\infty(M)\to C^\infty(\p M)$, $u\mapsto \p_\nu^j u|_{\p M}$, $j=0,1$,  extends to a linear continuous map
\begin{equation}
\label{lem_trace_map_1}
\gamma_j: H_{\Delta^2}(M)\to H^{-j-1/2}(\p M). 
\end{equation}

\item[(ii)] The trace map $\tilde \gamma_j: C^\infty(M)\to C^\infty(\p M)$, $u\mapsto \p_\nu^j (\Delta u)|_{\p M}$, $j=0,1$,  extends to a linear continuous map
\[
\tilde \gamma_j: H_{\Delta^2}(M)\to H^{-j-5/2}(\p M). 
\]
\end{itemize}

\end{lem}

\begin{proof}
We follow the arguments of \cite[Section 1]{Bukhgeim_Uhlmann_2002}, carried out in the case of $\Delta$.

(i). Let $j=0$, $u\in C^\infty(M)$, and $w\in H^{1/2}(\p M)$. By the Sobolev extension theorem, see \cite[Theorem 9.5]{Grubb_book},  there exists $v\in H^4(M^{\text{int}})$ such that 
\begin{equation}
\label{eq_2_2}
v|_{\p M}=0, \quad \p_\nu v|_{\p M}=0, \quad \p_\nu^2 v|_{\p M}=0,\quad   \p_\nu^3v|_{\p M}=w, 
\end{equation}
and 
\begin{equation}
\label{eq_2_3}
\|v\|_{H^4(M^{\text{int}})}\le C\|w\|_{H^{1/2}(\p M)}. 
\end{equation}
It follows from \eqref{eq_Green_1},  \eqref{eq_2_1}, \eqref{eq_2_2} that 
\[
 -\int_{\p M} u w dS= \int_M (\Delta^2 u) vdV-\int_M u(\Delta^2 v)dV,
\]
and therefore, using \eqref{eq_2_3}, we get 
\[
\bigg| \int_{\p M} u w dS\bigg|\le C\|u\|_{H_{\Delta^2}(M)}\|v\|_{H^4(M^{\text{int}})}\le C \|u\|_{H_{\Delta^2}(M)} \|w\|_{H^{1/2}(\p M)}.
\]
Hence, 
\begin{equation}
\label{eq_2_4}
\|\gamma_0 u\|_{H^{-1/2}(\p M)}\le C\|u\|_{H_{\Delta^2}(M)}. 
\end{equation}
By the density of  the space $C^\infty(M)$ in $H_{\Delta^2}(M)$, see \cite[Chapter 2, Section 8.1, page 192]{Lions_Magenes_1972},  and also  \cite[Theorem 9.8, and page 233]{Grubb_book}, we conclude that the map $\gamma_0$ extends to a continuous linear map: $H_{\Delta^2}(M)\to H^{-1/2}(\p M)$ and \eqref{eq_2_4} holds for all $u\in H_{\Delta^2}(M)$.  This shows (i) with $j=0$. 

Let next $j=1$ in (i) and let us now prove that $\gamma_1$ extends to a continuous linear map:  $H_{\Delta^2}(M)\to H^{-3/2}(\p M)$. To that end, let $u\in C^\infty(M)$ and let $w\in H^{3/2}(\p M)$. By the Sobolev extension theorem, there is $v\in H^4(M^{\text{int}})$ such that 
\begin{equation}
\label{eq_2_5}
v|_{\p M}=0, \quad \p_\nu v|_{\p M}=0, \quad \p_\nu^2 v|_{\p M}=w,\quad   \p_\nu^3v|_{\p M}=-Hw, 
\end{equation}
where $H$ is defined in \eqref{eq_2_1}, 
and
\begin{equation}
\label{eq_2_6}
\|v\|_{H^4(M^{\text{int}})}\le C\|w\|_{H^{3/2}(\p M)}. 
\end{equation}
It follows from \eqref{eq_2_1} and \eqref{eq_2_5} that 
\begin{equation}
\label{eq_2_7}
\Delta v|_{\p M}=w, \quad \p_\nu(\Delta v)|_{\p M}=0. 
\end{equation}
Using \eqref{eq_Green_1}, \eqref{eq_2_5}, \eqref{eq_2_7}, we get 
\[
\int_{\p M} (\p_\nu u) wdS =\int_M (\Delta^2 u) vdV-\int_M u(\Delta^2 v)dV,
\]
and therefore, using \eqref{eq_2_6}, we see that
\[
\bigg| \int_{\p M} (\p_\nu u) w dS\bigg|\le  C \|u\|_{H_{\Delta^2}(M)} \|w\|_{H^{3/2}(\p M)}.
\]
Thus, 
\begin{equation}
\label{eq_2_8}
\|\gamma_1 u\|_{H^{-3/2}(\p M)}\le C\|u\|_{H_{\Delta^2}(M)}. 
\end{equation}
By the density of  the space $C^\infty(M)$ in $H_{\Delta^2}(M)$, we obtain that the map $\gamma_1$ extends to a continuous linear map: $H_{\Delta^2}(M)\to H^{-3/2}(\p M)$ and \eqref{eq_2_8} holds for all $u\in H_{\Delta^2}(M)$.  This shows (i) with $j=1$.

(ii). The proof here follows along the same lines as in the case (i). Let us only mention that when $j=0$, we shall work with $w\in H^{5/2}(\p M)$ and $v\in H^4(M^{\text{int}})$ such that 
\[
v|_{\p M}=0, \quad \p_\nu v|_{\p M}= w,\quad \p_\nu^2 v=-Hw,\quad \p_\nu^3 v=-(\p_\nu H)w+H^2w-\Delta_t w.
\]
Therefore, this together with \eqref{eq_2_1} implies that 
\[
\Delta v|_{\p M}=0,\quad \p_\nu\Delta v|_{\p M}=0.
\]
We also have $\|v\|_{H^4(M^{\text{int}})}\le C\|w\|_{H^{5/2}(\p M)}$.

When  $j=1$, we shall work with $w\in H^{7/2}(\p M)$ and $v\in H^4(M^{\text{int}})$ such that 
\[
v|_{\p M}=w, \quad \p_\nu v|_{\p M}= 0,\quad \p_\nu^2 v=-\Delta_t w,\quad \p_\nu^3 v=H\Delta_t w.
\]
Therefore, by \eqref{eq_2_1}, we get  
\[
\Delta v|_{\p M}=0,\quad \p_\nu\Delta v|_{\p M}=0.
\]
We also have $\|v\|_{H^4(M^{\text{int}})}\le C\|w\|_{H^{7/2}(\p M)}$. This completes the proof of Lemma \ref{lem_trace_map}. 
\end{proof}

By Lemma \ref{lem_trace_map},  we have the following consequence of  \eqref{eq_Green_1}. 
\begin{cor}
For any $u\in H_{\Delta^2}(M)$ and $v\in H^4(M^{\text{int}})$, we have the following generalized Green formula, \begin{equation}
\label{eq_Green}
\begin{aligned}
\int_M (\Delta^2 u) vdV-\int_M u\Delta^2 vdV=&\int_{\p M} \p_\nu u (\Delta v)dS -\int_{\p M} u\p_\nu(\Delta v)dS\\
&+\int_{\p M} \p_\nu(\Delta u) vdS-\int_{\p \Omega} (\Delta u)\p_\nu vdS,
\end{aligned}
\end{equation}
where 
\begin{align*}
&\int_{\p M} \p_\nu u (\Delta v)dS:=\langle \gamma_1 u, \Delta v \rangle_{H^{-3/2}(\p M), H^{3/2}(\p M)},  \\ 
&\int_{\p M} u\p_\nu(\Delta v)dS:=\langle \gamma_0 u, \p_\nu(\Delta v) \rangle_{H^{-1/2}(\p M), H^{1/2}(\p M)},\\
&\int_{\p M} \p_\nu(\Delta u) vdS:= \langle \tilde  \gamma_1 u, v \rangle_{H^{-7/2}(\p M), H^{7/2}(\p M)}, \\
&\int_{\p \Omega} (\Delta u)\p_\nu vdS:= \langle \tilde  \gamma_0 u, \p_\nu v \rangle_{H^{-5/2}(\p M), H^{5/2}(\p M)}.
\end{align*}
\end{cor}

We shall need the following extension of \cite[Theorem 26.3]{Eskin_book} to the case of the biharmonic operator $\Delta^2$. Here for $u\in H_{\Delta^2}(M)$, we set
\begin{equation}
\label{eq_2_9_-1}
\gamma u= (\gamma_0 u,\gamma_1 u),
\end{equation} 
where $\gamma_j$, $j=0,1$, are given by 
\eqref{lem_trace_map_1}. Note $\gamma$ in \eqref{eq_2_9_-1} is an extension of the trace map in   \eqref{int_trace}. 
\begin{thm}
\label{thm_eskin}
For each $g=(g_0, g_1)\in H^{-1/2}(\p M)\times H^{-3/2}(\p M)$, there exists a unique $u\in L^2(M)$ such that 
\begin{equation}
\label{eq_2_9}
\begin{cases}
\Delta^2 u=0& \text{in}\quad M^{\text{int}},\\
\gamma u=g& \text{on}\quad \p M,
\end{cases}
\end{equation}
and 
\begin{equation}
\label{eq_2_10}
\|u\|_{L^2(M)}\le C\|g\|_{H^{-1/2}(\p M)\times H^{-3/2}(\p M)}. 
\end{equation}
Here $\|g\|_{H^{-1/2}(\p M)\times H^{-3/2}(\p M)}^2=\|g_0\|_{H^{-1/2}(\p M)}^2+ \|g_1\|_{H^{-3/2}(\p M)}^2$. 
\end{thm}

\begin{proof}
We shall follow the proof of  \cite[Theorem 26.3]{Eskin_book}.  Let $v\in H^4(M^{\text{int}})$ be such that $v|_{\p M}=0$, $\p_\nu v|_{\p M}=0$. If there is $u\in L^2(M)$ satisfying \eqref{eq_2_9} then by the generalized Green formula \eqref{eq_Green}, we obtain 
\begin{equation}
\label{eq_2_11}
\int_M u\Delta^2 vdV=\langle g_0, \p_\nu(\Delta v) \rangle_{H^{-1/2}(\p M), H^{1/2}(\p M)}-\langle g_1, \Delta v \rangle_{H^{-3/2}(\p M), H^{3/2}(\p M)}.
\end{equation}
Consider the subspace
\[
L:=\{\Delta^2 v: v\in H^4(M^{\text{int}}), v|_{\p M}=0, \ \p_\nu v|_{\p M}=0 \}\subset L^2(M).
\]
In view of \eqref{eq_2_11}, we define the linear functional $F$ on $L$ by
\begin{equation}
\label{eq_2_12}
F(\Delta^2 v):=\langle g_0, \p_\nu(\Delta v) \rangle_{H^{-1/2}(\p M), H^{1/2}(\p M)}-\langle g_1, \Delta v \rangle_{H^{-3/2}(\p M), H^{3/2}(\p M)}.
\end{equation}
Using the Cauchy--Schwarz inequality, the following Sobolev trace theorem
\[
\|(v,\p_\nu v, \p_\nu^2 v, \p_\nu^3 v)\|_{(H^{7/2}\times H^{5/2}\times H^{3/2}\times H^{1/2})(\p M)}\le C\|v\|_{H^4(M^{\text{int}})}, 
\]
and \eqref{eq_2_1}, we obtain from \eqref{eq_2_12} that 
\begin{equation}
\label{eq_2_13}
\begin{aligned}
|F(\Delta^2 v)|&\le \|g_0\|_{H^{-1/2}(\p M)}\|\p_\nu(\Delta v)\|_{H^{1/2}(\p M)}+\|g_1\|_{H^{-3/2}(\p M)}\|\Delta v\|_{H^{3/2}(\p M)}\\
&\le C\|g\|_{H^{-1/2}(\p M)\times H^{-3/2}(\p M)}\|v\|_{H^4(M^{\text{int}})}.
\end{aligned}
\end{equation}
Using the fact that $v|_{\p M}=0$,  $\p_\nu v|_{\p M}=0$, and boundary elliptic regularity, see \cite[Theorem 11.14]{Grubb_book}, we get 
\begin{equation}
\label{eq_2_14}
\|v\|_{H^4(M^{\text{int}})}\le C\|\Delta^2 v\|_{L^2(M)}. 
\end{equation}

Combining \eqref{eq_2_13} and \eqref{eq_2_14}, we obtain that 
\[
|F(\Delta^2 v)|\le C\|g\|_{H^{-1/2}(\p M)\times H^{-3/2}(\p M)}\|\Delta^2 v\|_{L^2(M)},
\]
which shows that $F$ is bounded on $L$. Thus, by the Hahn-Banach theorem, $F$ can be extended to a bounded linear functional on $L^2(M)$, and by Riesz representation theorem, there exists $u\in L^2(M)$ such that 
\begin{equation}
\label{eq_2_15}
F(\Delta^2 v)=\int_M (\Delta^2 v) udV,
\end{equation}
and \eqref{eq_2_10} holds.  Letting $v\in C^\infty_0(M^{\text{int}})$, we conclude from \eqref{eq_2_15} and \eqref{eq_2_12} that $\Delta^2 u=0$ in $M^{\text{int}}$.

Using \eqref{eq_2_15}, \eqref{eq_2_12}, and the generalized Green formula \eqref{eq_Green}, we get 
\begin{equation}
\label{eq_2_16}
\begin{aligned}
&\langle \gamma_0 u, \p_\nu(\Delta v) \rangle_{H^{-1/2}(\p M), H^{1/2}(\p M)}-\langle \gamma_1 u, \Delta v \rangle_{H^{-3/2}(\p M), H^{3/2}(\p M)}\\
&=\langle g_0, \p_\nu(\Delta v) \rangle_{H^{-1/2}(\p M), H^{1/2}(\p M)}-\langle g_1, \Delta v \rangle_{H^{-3/2}(\p M), H^{3/2}(\p M)},
\end{aligned}
\end{equation}
for all $v\in H^4(M^{\text{int}})$  such that $v|_{\p M}=0$, $\p_\nu v|_{\p M}=0$.

Letting $w\in H^{1/2}(\p M)$, and taking $v\in H^4(M^{\text{int}})$ such that \eqref{eq_2_2} holds, we see from  \eqref{eq_2_16} that $\gamma_0 u=g_0$. Furthermore, letting $w\in H^{3/2}(\p M)$ and taking  $v\in H^4(M^{\text{int}})$ such that
\eqref{eq_2_5} holds, in view of  \eqref{eq_2_7}, we conclude from \eqref{eq_2_16} that $\gamma_1 u=g_1$.

The uniqueness follows from the fact that if $u\in L^2(M)$ solves the Dirichlet problem \eqref{eq_2_9} with $g=0$ then by the boundary elliptic regularity, see \cite[Theorem 11.14]{Grubb_book}, $u\in H^4(M^{\text{int}})$, and therefore, $u=0$. 
\end{proof}

\begin{cor}
\label{cor_trace_bijective}
Let $q\in C(M)$ be such that assumption (A) is satisfied, and let 
\[
H_q:=\{u\in L^2(M): (\Delta^2+q)u=0\}\subset H_{\Delta^2}(M). 
\]
Then the trace map
\begin{equation}
\label{eq_2_17}
\gamma: H_q\to H^{-1/2}(\p M)\times H^{-3/2}(\p M)
\end{equation}
is bijective.
\end{cor}
\begin{proof}
We begin by showing that the map $\gamma$ in \eqref{eq_2_17} is surjective. To that end, letting $g\in H^{-1/2}(\p M)\times H^{-3/2}(\p M)$, by Theorem \ref{thm_eskin}, we get  a unique $u\in L^2(M)$ satisfying \eqref{eq_2_9}. Assumption (A) implies that there is a unique $v\in H^2_0(M^{\text{int}})$ such that 
\begin{equation}
\label{eq_2_18}
\begin{cases}
(\Delta^2+q) v=qu& \text{in}\quad M^{\text{int}},\\
\gamma v=0& \text{on}\quad \p M. 
\end{cases}
\end{equation}
Now letting $w=u-v\in L^2(M)$, in view of \eqref{eq_2_9} and \eqref{eq_2_18}, we see that $w\in H_q$ and $\gamma w=g$. This shows the surjectivity of $\gamma$ in  \eqref{eq_2_17}.

The injectivity of  $\gamma$ in \eqref{eq_2_17}  follows from the fact that if $u\in H_q$ is such that $\gamma u=0$ then the boundary elliptic regularity, see \cite[Theorem 11.14]{Grubb_book},  shows that $u\in (H^4\cap H^2_0)(M^{\text{int}})$, and by assumption (A), $u=0$. 
\end{proof}

In view of Corollary \ref{cor_trace_bijective}, we can define the  Poisson operator as follows,
\begin{equation}
\label{eq_2_19}
\mathcal{P}_q=\gamma^{-1}: H^{-1/2}(\p M)\times H^{-3/2}(\p M)\to H_q.
\end{equation}
We have 
\begin{equation}
\label{eq_2_18_1}
\|\mathcal{P}_q f\|_{L^2(M)}\le  C\|f\|_{H^{-1/2}(\p M)\times H^{-3/2}(\p M)}, 
\end{equation}
for all $f\in H^{-1/2}(\p M)\times H^{-3/2}(\p M)$. 

Finally, let us derive the integral identity which will be used to reconstruct the potential. To that end, let $f,g\in H^{3/2}(\p M)\times H^{1/2}(\p M)$, let $u=u^{f}\in H^2(M^{\text{int}})$ be the unique solution to 
the Dirichlet problem
\begin{equation}
\label{eq_2_20}
\begin{cases}
(\Delta^2+q)u=0 & \text{in}\quad M^{\text{int}},\\
\gamma u=f & \text{on}\quad \p M,
\end{cases}
\end{equation}
and let $v=v^{f}\in H^2(M^{\text{int}})$ be the unique solution to 
the Dirichlet problem
\begin{equation}
\label{eq_2_21}
\begin{cases}
\Delta^2v=0 & \text{in}\quad M^{\text{int}},\\
\gamma v=g & \text{on}\quad \p M.
\end{cases}
\end{equation}
By the definition of the Dirichlet--to--Neumann map \eqref{int_2_DN_map}, we get
\begin{equation}
\label{eq_2_22}
\langle \Lambda_q f,g \rangle_{H^{-3/2}(\p M)\times H^{-1/2}(\p M), H^{3/2}(\p M)\times H^{1/2}(\p M)}=\int_M(\Delta u^f)(\Delta v^g)dV+\int_M qu^fv^gdV, 
\end{equation}
and 
\begin{equation}
\label{eq_2_23}
\begin{aligned}
\langle \Lambda_0 g,&f \rangle_{H^{-3/2}(\p M)\times H^{-1/2}(\p M), H^{3/2}(\p M)\times H^{1/2}(\p M)}=\int_M(\Delta v^g)(\Delta u^f)dV\\
&=\int_M(\Delta v^g)(\Delta v^f)dV=\langle \Lambda_0 f,g \rangle_{H^{-3/2}(\p M)\times H^{-1/2}(\p M), H^{3/2}(\p M)\times H^{1/2}(\p M)}.
\end{aligned} 
\end{equation}
In the penultimate equality of \eqref{eq_2_23} we used the fact that the definition of the Dirichlet--to--Neumann map $\Lambda_0$ is independent of the choice of extension of $f\in H^{3/2}(\p M)\times H^{1/2}(\p M)$ to an $H^2(M^{\text{int}})$ element whose trace is equal to $f$. Considering the difference of  \eqref{eq_2_22} and \eqref{eq_2_23}, we obtain the following integral identity, 
\begin{equation}
\label{eq_2_24}
\langle (\Lambda_q-\Lambda_0) f,g \rangle_{H^{-3/2}(\p M)\times H^{-1/2}(\p M), H^{3/2}(\p M)\times H^{1/2}(\p M)}=\int_M quvdV, 
\end{equation}
where $u=u^f, v=v^g\in H^2(M^{\text{int}})$ are solutions to \eqref{eq_2_20} and \eqref{eq_2_21}, respectively. 

We would like to extend the Nachman--Street argument \cite{Nachman_Street_2010} to reconstruct the potential $q$ from the knowledge of the Dirichlet--to--Neumann map for the biharmonic operator and therefore, as in \cite{Nachman_Street_2010}, we shall work with $L^2(M)$ solutions rather than $H^2(M^{\text{int}})$ solutions to the Dirichlet problems \eqref{eq_2_20},  \eqref{eq_2_21}. Thus, we shall need to extend the integral identity \eqref{eq_2_24} to such solutions. In doing so, we first claim that $\Lambda_q-\Lambda_0$ extends to a linear continuous map
\begin{equation}
\label{eq_2_25}
\Lambda_q-\Lambda_0: H^{-1/2}(\p M)\times H^{-3/2}(\p M) \to H^{1/2}(\p M)\times H^{3/2}(\p M).
\end{equation}
To that end, letting $f,g\in C^\infty(\p M)\times C^\infty(\p M)$, we conclude from \eqref{eq_2_24}, \eqref{eq_2_10}, and \eqref{eq_2_18_1} that 
\begin{align*}
|\langle (\Lambda_q-\Lambda_0) f,g \rangle_{L^2(\p M)\times L^2(\p M), L^2(\p M)\times L^2(\p M)}|\le C\|u\|_{L^2(M)}\|v\|_{L^2(M)}\\
\le C\|f\|_{H^{-1/2}(\p M)\times H^{-3/2}(\p M)} \|g\|_{H^{-1/2}(\p M)\times H^{-3/2}(\p M)}. 
\end{align*}
Hence, 
\[
\|(\Lambda_q-\Lambda_0) f\|_{H^{1/2}(\p M)\times H^{3/2}(\p M)}\le C\|f\|_{H^{-1/2}(\p M)\times H^{-3/2}(\p M)},
\]
which together with the density of $C^\infty(\p M)\times C^\infty(\p M)$ in the space $H^{-1/2}(\p M)\times H^{-3/2}(\p M)$ gives the claim \eqref{eq_2_25}.

Now letting $f,g\in H^{-1/2}(\p M)\times H^{-3/2}(\p M)$, approximating them by $C^\infty(\p M)\times C^\infty(\p M)$--functions, using \eqref{eq_2_25}, \eqref{eq_2_10}, and \eqref{eq_2_18_1}, we obtain from \eqref{eq_2_24} that 
\begin{equation}
\label{eq_2_26}
\langle (\Lambda_q-\Lambda_0) f,g \rangle_{H^{1/2}(\p M)\times H^{3/2}(\p M), H^{-1/2}(\p M)\times H^{-3/2}(\p M)}=\int_M quvdV, 
\end{equation}
where $u=u^f, v=v^g\in L^2(M)$ are solutions to \eqref{eq_2_20} and \eqref{eq_2_21}, respectively.

\section{The Nachman--Street argument for biharmonic operators}
\label{sec_Nachman-Street}

The goal of this section is to extend the Nachman--Street argument \cite{Nachman_Street_2010} for constructive determination of the boundary traces of  suitable complex geometric optics solutions, developed for the Schr\"odinger equation, to the case of the perturbed biharmonic equation. Specifically, we shall extend to the case of the perturbed biharmonic equation  the simplified version of the Nachman--Street argument, presented in \cite{Feizm_K_Oksan_Uhl} in the full data case in the setting of compact Riemannian manifolds with boundary admitting a limiting Carleman weight.

Let $(M,g)$ be a smooth compact Riemannian manifold of dimension $n\ge 3$ with smooth boundary $\p M$, and let $-h^2\Delta_g=-h^2\Delta$  be the semiclassical Laplace--Beltrami operator on $M$, where $h>0$ is a small semiclassical parameter.  Assume, as we may,  that $(M,g)$ is embedded in a compact smooth Riemannian manifold $(N,g)$ without boundary of the same dimension, and let $U$ be open in $N$ such that $M\subset U$. 
When $\varphi\in C^\infty(U;\R)$, we let 
 \[
P_\varphi=e^{\frac{\varphi}{h}}(-h^2\Delta)e^{-\frac{\varphi}{h}}
\]
be  the conjugated operator, and let $p_\varphi$ be its semiclassical principal symbol.  Following \cite{Kenig_Sjostrand_Uhlmann},  \cite{DKSaloU_2009}, we say that $\varphi\in C^\infty(U;\R)$ is a limiting Carleman weight for $-h^2\Delta$ on $(U,g)$ if $d\varphi\ne 0$ on $U$, and the Poisson bracket of $\Re p_\varphi$ and $\Im p_\varphi$ satisfies,
\[
\{\Re p_\varphi, \Im p_\varphi\}=0 \quad\text{when}\quad  p_\varphi=0. 
\]

Using Carleman estimates for $-h^2\Delta$, established in  \cite{DKSaloU_2009}, it was shown in \cite{Nachman_Street_2010}, see also \cite[Proposition 2.2]{Feizm_K_Oksan_Uhl},  that for all $0<h\ll 1$ and any $v\in L^2(M)$, there exists a unique solution $u\in  (\text{Ker} (P_\varphi))^\perp $ of the equation
\[
P_\varphi u=v\quad \text{in}\quad M^{\text{int}}.
\]
Here 
\[
\text{Ker}(P_\varphi)=\{u\in L^2(M): P_\varphi u=0\}.
\]
Based on this unique solution, the Green operator $G_\varphi$ for $P_\varphi$ was constructed in \cite{Nachman_Street_2010}, see also \cite[Theorem 2.3]{Feizm_K_Oksan_Uhl}, enjoying the following properties: for all $0<h\ll 1$, there exists a linear continuous operator $G_{\varphi}:L^2(M)\to L^2(M)$ such that 
\begin{equation}
\label{eq_3_1}
\begin{aligned}
&P_{\varphi} G_{ \varphi}=I  \text{ on }  L^2(M),  \quad
\|G_{ \varphi}\|_{\mathcal{L}(L^2(M), L^2(M))}=\mathcal{O}(h^{-1}), \\
&G_\varphi^*=G_{-\varphi},\quad 
G_{ \varphi} P_{\varphi}=I \text{ on }C^\infty_0(M^{\text{int}}).
 \end{aligned}
\end{equation}
Here $G_\varphi^*$ denotes  the $L^2(M)$--adjoint of $G_\varphi$.  Letting $P_\varphi^*$ be the formal $L^2(M)$--adjoint of $P_\varphi$, we see that $P_\varphi^*=P_{-\varphi}$.  Note also that  if $\varphi$ is a limiting Carleman weight for $-h^2\Delta$ then so is $-\varphi$. 

In this paper we shall work with the semiclassical biharmonic operator $(-h^2\Delta)^2$. We have 
 \[
P_\varphi^2=e^{\frac{\varphi}{h}}(-h^2\Delta)^2e^{-\frac{\varphi}{h}}.
\]
We shall use $G^2_\varphi: L^2(M)\to L^2(M)$ as Green's operator for $P_\varphi^2$. It follows from \eqref{eq_3_1} that $G^2_\varphi$ enjoys the following properties, 
\begin{equation}
\label{eq_3_2}
\begin{aligned}
&P_{\varphi}^2 G^2_{ \varphi}=I  \text{ on }  L^2(M),  \quad
\|G^2_{ \varphi}\|_{\mathcal{L}(L^2(M), L^2(M))}=\mathcal{O}(h^{-2}), \\
&(G^2_\varphi)^*=G^2_{-\varphi},\quad 
G^2_{ \varphi} P^2_{\varphi}=I \text{ on }C^\infty_0(M^{\text{int}}).
 \end{aligned}
\end{equation}
Furthermore, the first identity in \eqref{eq_3_2} implies that 
\begin{equation}
\label{eq_3_3}
G_\varphi^2:L^2(M)\to e^{\varphi/h}H_{\Delta^2}(M). 
\end{equation}

Next we shall proceed to introduce single layer operators associated to the Green operator $G_\varphi^2$. 
First note that the trace map $\gamma$ given by \eqref{eq_2_9_-1} has the following mapping properties,
\begin{equation}
\label{eq_3_3_1}
\gamma: e^{\pm\varphi/h}H_{\Delta^2(M)}\to e^{\pm\varphi/h}(H^{-1/2}(\p M)\times H^{-3/2}(\p M))=H^{-1/2}(\p M)\times H^{-3/2}(\p M),
\end{equation}
and therefore, using \eqref{eq_3_3}, we get 
\[
\gamma\circ G_\varphi^2:L^2(M)\to H^{-1/2}(\p M)\times H^{-3/2}(\p M)
\]
is continuous. This implies that the $L^2$--adjoint 
\begin{equation}
\label{eq_3_4_0}
(\gamma\circ G_\varphi^2)^*: H^{1/2}(\p M)\times H^{3/2}(\p M)\to L^2(M)
\end{equation}
is also continuous.  For any $g\in H^{1/2}(\p M)\times H^{3/2}(\p M)$, we have
\begin{equation}
\label{eq_3_4}
P_{-\varphi}^2((\gamma\circ G_\varphi^2)^*g)=0 \quad \text{in}\quad \mathcal{D}'(M^{\text{int}}).
\end{equation}
The proof is based on the following observation. Letting $f\in C^\infty_0(M^{\text{int}})$, using the fourth property in \eqref{eq_3_2}, we get 
\begin{align*}
(P_{-\varphi}^2&((\gamma\circ G_\varphi^2)^*g), f)_{L^2(M)}=((\gamma\circ G_\varphi^2)^*g, P_{\varphi}^2 f)_{L^2(M)}\\
&= (g, (\gamma\circ G_\varphi^2)P_{\varphi}^2 f)_{H^{1/2}(\p M)\times H^{3/2}(\p M),H^{-1/2}(\p M)\times H^{-3/2}(\p M)}=0.
\end{align*}

Now \eqref{eq_3_4} gives us the following mapping properties for the operator $(\gamma\circ G_\varphi^2)^*$, 
\[
(\gamma\circ G_\varphi^2)^*: H^{1/2}(\p M)\times H^{3/2}(\p M)\to e^{-\varphi/h}H_{\Delta^2}(M),
\]
which improves \eqref{eq_3_4_0}. Thus, in view of \eqref{eq_3_3_1}, we have that the map
\[
\gamma\circ(\gamma\circ G_\varphi^2)^*: H^{1/2}(\p M)\times H^{3/2}(\p M)\to H^{-1/2}(\p M)\times H^{-3/2}(\p M)
\]
is well defined and continuous, and therefore, its $L^2$--adjoint 
\[
(\gamma\circ(\gamma\circ G_\varphi^2)^*)^*: H^{1/2}(\p M)\times H^{3/2}(\p M)\to H^{-1/2}(\p M)\times H^{-3/2}(\p M)
\]
is also continuous.  We introduce the single layer operator associated to the Green operator $G_\varphi^2$ as follows: 
\begin{equation}
\label{eq_3_5}
\begin{aligned}
S_\varphi=&e^{-\varphi/h}(\gamma\circ(\gamma\circ G_\varphi^2)^*)^* e^{\varphi/h}\\
&\in \mathcal{L}(H^{1/2}(\p M)\times H^{3/2}(\p M), H^{-1/2}(\p M)\times H^{-3/2}(\p M)).
\end{aligned}
\end{equation}
Note that  definition \eqref{eq_3_5} looks similar to the corresponding single layer operator in the case of the Laplacian in \cite{Nachman_Street_2010}, see also \cite{Feizm_K_Oksan_Uhl}, with the only difference that here the Green operator is  $G_\varphi^2$ instead of $G_\varphi$ and the trace $\gamma$ has two components. 

Now in view of \eqref{eq_2_25} and \eqref{eq_3_5}, we have 
\[
S_\varphi(\Lambda_q-\Lambda_0): H^{-1/2}(\p M)\times H^{-3/2}(\p M) \to H^{-1/2}(\p M)\times H^{-3/2}(\p M).
\]
is continuous. We claim that 
\begin{equation}
\label{eq_3_6}
S_\varphi(\Lambda_q-\Lambda_0)=\gamma\circ e^{-\varphi/h}\circ G_\varphi^2\circ e^{\varphi/h}\circ q\circ \mathcal{P}_q
\end{equation}
in the sense of linear continuous operators on the space  $H^{-1/2}(\p M)\times H^{-3/2}(\p M)$. Here $\mathcal{P}_q$ is the Poisson operator given by \eqref{eq_2_19}. To see \eqref{eq_3_6}, letting $f,g\in C^\infty(\p M)\times C^\infty(\p M)$, we get 
\begin{align*}
\langle \gamma\circ & e^{-\varphi/h}\circ G_\varphi^2\circ e^{\varphi/h}\circ q\circ \mathcal{P}_q f, g\rangle_{H^{-1/2}(\p M)\times H^{-3/2}(\p M), H^{1/2}(\p M)\times H^{3/2}(\p M)}\\
&=
\langle q\circ \mathcal{P}_q f,  e^{\varphi/h}(\gamma\circ G_\varphi^2)^*e^{-\varphi/h}g\rangle_{L^2(M), L^2(M)}\\
&=\langle (\Lambda_q-\Lambda_0) f,\gamma\circ e^{\varphi/h}(\gamma\circ G_\varphi^2)^*e^{-\varphi/h}g\rangle_{H^{-1/2}(\p M)\times H^{-3/2}(\p M), H^{1/2}(\p M)\times H^{3/2}(\p M)}\\
&=\langle  S_\varphi(\Lambda_q-\Lambda_0) f, g\rangle_{H^{-1/2}(\p M)\times H^{-3/2}(\p M), H^{1/2}(\p M)\times H^{3/2}(\p M)},
\end{align*}
showing \eqref{eq_3_6}. 
Here in the penultimate equality, we used the fact that $\Delta^2(e^{\varphi/h}(\gamma\circ G_\varphi^2)^*e^{-\varphi/h}g)=0$ in $M^{\text{int}}$ in view of \eqref{eq_3_4} and the integral identity \eqref{eq_2_26}, and in the last equality we used \eqref{eq_3_5}.

Similar to \cite[Proposition 2.4]{Feizm_K_Oksan_Uhl}, we have the following result. 
\begin{prop}
\label{prop_2_4} 
Let $f,g\in H^{-1/2}(\p M)\times H^{-3/2}(\p M)$. Then 
\begin{equation}
\label{eq_3_7}
(1+h^4S_\varphi(\Lambda_q-\Lambda_0))f=g
\end{equation}
if and only if 
\begin{equation}
\label{eq_3_8}
(1+e^{-\varphi/h}\circ G_\varphi^2\circ e^{\varphi/h} h^4q)\mathcal{P}_q f=\mathcal{P}_0 g.
\end{equation}
\end{prop}
\begin{proof}
Assume first that \eqref{eq_3_7} holds. To show that \eqref{eq_3_8} holds, we first observe that $(h^2\Delta)^2\mathcal{P}_q f=-h^4q\mathcal{P}_q f$. Using the first property in \eqref{eq_3_2}, we also obtain that 
\begin{equation}
\label{eq_3_9}
(h^2\Delta)^2(1+e^{-\varphi/h}\circ G_\varphi^2\circ e^{\varphi/h} h^4q)\mathcal{P}_q f=0 \quad \text{in}\quad M^{\text{int}}.
\end{equation} 
Furthermore, \eqref{eq_3_6} and \eqref{eq_3_7} imply that 
\begin{equation}
\label{eq_3_9_1}
\gamma (1+e^{-\varphi/h}\circ G_\varphi^2\circ e^{\varphi/h} h^4q)\mathcal{P}_q f=f+ h^4S_\varphi(\Lambda_q-\Lambda_0))f=g.
\end{equation}
By the uniqueness result of Theorem \ref{thm_eskin} applied to \eqref{eq_3_9} and \eqref{eq_3_9_1}, we obtain \eqref{eq_3_8}. 

Now if \eqref{eq_3_8} holds then \eqref{eq_3_7} can be obtained by taking the trace $\gamma$ on both sides of \eqref{eq_3_8}. 
\end{proof}

The recovery of the boundary traces of suitable complex geometric optics solutions to the equation $(\Delta^2+q)u=0$ will be based on the following result, which is  similar to \cite[Proposition 2.5]{Feizm_K_Oksan_Uhl}. 
\begin{prop}
\label{prop_2_5} 
The operator  $1+h^4S_\varphi(\Lambda_q-\Lambda_0):H^{-1/2}(\p M)\times H^{-3/2}(\p M)\to H^{-1/2}(\p M)\times H^{-3/2}(\p M)$ is a linear homemorphism for all $0<h\ll 1$. 
\end{prop}

\begin{proof}
First using that $\|G_\varphi^2\|_{L^2(M)\to L^2(M)}=\mathcal{O}(h^{-2})$, see \eqref{eq_3_2}, we observe that the operator $1+e^{-\varphi/h}\circ G_\varphi^2\circ e^{\varphi/h} h^4q$ in \eqref{eq_3_8} is a linear homemorphism on $L^2(M)$ for all $0<h\ll 1$. Thus, for all $0<h\ll 1$ and for all $v\in L^2(M)$, the equation 
\[
(1+e^{-\varphi/h}\circ G_\varphi^2\circ e^{\varphi/h} h^4q)u=v \quad \text{in}\quad M^{\text{int}}
\]
has a unique solution $u\in L^2(M)$. Furthermore, if $v\in H_0$ then $u\in H_q$ by the first property of \eqref{eq_3_2}. Hence, for all $0<h\ll 1$, the operator  $1+e^{-\varphi/h}\circ G_\varphi^2\circ e^{\varphi/h} h^4q: H_q\to H_0$ is an isomorphism.  
It follows from \eqref{eq_2_19} that the operator $(1+e^{-\varphi/h}\circ G_\varphi^2\circ e^{\varphi/h} h^4q)\circ \mathcal{P}_q: H^{-1/2}(\p M)\times H^{-3/2}(\p M) \to H_0$ is an isomorphism for all $0<h\ll 1$. This together with Proposition \ref{prop_2_4} implies the claim. 
\end{proof}

\section{Proof of Theorem \ref{thm_main}}
\label{section_4}

Let $(M,g)$ be a CTA manifold so that $(M,g)\subset \subset (\R\times M_0^{\text{int}}, c(e\oplus g_0))$. Since $(M,g)$ is known, the transversal manifold $(M_0,g_0)$ as well as the conformal factor $c$ are also known. Therefore, the Dirichlet--to--Neumann map $\Lambda_0$ is also known.  Furthermore, we assume the knowledge of the Dirichlet--to--Neumann map $\Lambda_q$. Using the integral identity \eqref{eq_2_26}, we would like to reconstruct the potential $q$ from this data. 

Let $x=(x_1,x')$ be the local coordinates in $\R\times M_0$. We know from \cite{DKSaloU_2009}  that the function $\varphi(x)=x_1$ is a limiting Carleman weight for the semiclassical Laplacian $-h^2\Delta$. Our starting point is the following result about the existence of Gaussian beam quasimodes for the biharmonic operator, constructed on $M$ and localized to non-tangential geodesics on  the transversal manifold  $M_0$ times $\R_{x_1}$, established in \cite[Propositions 2.1, 2.2]{Yan_2020}. See also \cite{Babich_Buldyrev}, \cite{Ralston_1977}, \cite{Ralston_1982}, \cite{DKurylevLS_2016}, \cite{Krup_Liimatainen_Salo} for related constructions of Gaussian beam quasimodes for second order operators and applications to inverse boundary problems. 
\begin{thm}
\label{thm_Lili}
Let $s=\frac{1}{h}+i\lambda$, $0<h<1$, $\lambda\in \R$ and  let $\gamma:[0,L]\to M_0$ be a unit speed non-tangential geodesic on $M_0$. Then there are families of Gaussian beam quasimodes $v_s,w_s\in C^\infty(M)$ such that 
\begin{equation}
\label{eq_4_1}
\|v_s\|_{H^1_{\emph{\text{scl}}}(M^\emph{\text{int}})}=\mathcal{O}(1),\quad \|e^{sx_1}(h^2\Delta)^2 e^{-sx_1}v_s\|_{L^2(M)}=\mathcal{O}(h^{5/2}),
\end{equation}
\begin{equation}
\label{eq_4_2}
\|w_s\|_{H^1_{\emph{\text{scl}}}(M^\emph{\text{int}})}=\mathcal{O}(1),\quad \|e^{-sx_1}(h^2\Delta)^2 e^{sx_1}w_s\|_{L^2(M)}=\mathcal{O}(h^{5/2}),
\end{equation}
as $h\to 0$.  Furthermore, letting $\psi\in C(M_0)$, and letting $x_1\in \R$, we have
\begin{equation}
\label{eq_4_3}
\lim_{h\to 0} \int_{\{x_1\}\times M_0} v_s\overline{w_s}\psi dV_{g_0}=\int_0^L e^{-2\lambda t} c(x_1,\gamma(t))^{1-\frac{n}{2}}\psi(\gamma(t))dt. 
\end{equation}
\end{thm}

We shall use the Gaussian beam quasimodes of Theorem \ref{thm_Lili} to construct solutions $u_2, u_1\in L^2(M)$ to the biharmonic equation $\Delta^2 u_2=0$ and the perturbed biharmonic equation $(\Delta^2+q)u_1=0$ in $M$, which will be used to test the integral identity \eqref{eq_2_26}. Note that some solutions of the perturbed biharmonic equations based on the Gaussian beam quasimodes of Theorem \ref{thm_Lili} were constructed in \cite{Yan_2020}. Here our construction will  be different as we need to be able to reconstruct their traces $\gamma u_1=(u_1|_{\p M}, \p_\nu u_1|_{\p M})$. Specifically, our construction will be based on the Green operator $G^2_\varphi$ for the conjugated biharmonic operator $P^2_\varphi$. 

First,  let us define $u_2\in L^2(M)$  by 
\begin{equation}
\label{eq_4_4}
u_2=e^{sx_1}(w_s+\tilde r_2),
\end{equation}
where $w_s$ is the Gaussian beam quasimode given by Theorem \ref{thm_Lili} and $\tilde r_2\in L^2(M)$ is the remainder term. 
Now $u_2$ solves $\Delta^2 u_2=0$ if $\tilde r_2$ satisfies
\begin{equation}
\label{eq_4_5}
P^2_{-\varphi} e^{i\lambda x_1}\tilde r_2=-e^{i\lambda x_1}e^{-sx_1}h^4\Delta^2e^{sx_1}w_s.
\end{equation}
Looking for $\tilde r_2$ in the form $\tilde r_2=e^{-i\lambda x_1}G_{-\varphi}^2 r_2$ with $r_2\in L^2(M)$, we see from \eqref{eq_4_5} and \eqref{eq_3_2} that $r_2=-e^{i\lambda x_1}e^{-sx_1}h^4\Delta^2e^{sx_1}w_s$. It follows from \eqref{eq_4_2} that 
$\|r_2\|_{L^2(M)}=\mathcal{O}(h^{5/2})$, and therefore, using \eqref{eq_3_2}, we get 
\begin{equation}
\label{eq_4_6}
\|\tilde r_2\|_{L^2(M)}=\mathcal{O}(h^{1/2}),
\end{equation}
as $h\to 0$. 

Next we look for $u_1\in L^2(M)$ solving 
\begin{equation}
\label{eq_4_6_1}
(\Delta^2+q)u_1=0\quad \text{in}\quad  M^{\text{int}}
\end{equation}
 in the form,
\begin{equation}
\label{eq_4_7}
u_1=u_0+e^{-sx_1}\tilde r_1.
\end{equation}
Here $u_0\in L^2(M)$ is such that 
\begin{equation}
\label{eq_4_8}
\Delta^2 u_0=0\quad\text{in}\quad  M^{\text{int}},
\end{equation}
 and $u_0$ has the form, 
\begin{equation}
\label{eq_4_9}
u_0=e^{-sx_1}(v_s+\tilde r_0),
\end{equation}
where $v_s$ is the Gaussian beam quasimode given by Theorem \ref{thm_Lili}, and $\tilde r_0, \tilde r_1\in L^2(M)$ are the remainder terms. First in view of \eqref{eq_4_8}, $\tilde r_0$ should satisfy 
\begin{equation}
\label{eq_4_10}
P_\varphi^2 e^{-i\lambda x_1}\tilde r_0=-e^{-i\lambda x_1}e^{sx_1}h^4\Delta^2 e^{-sx_1}v_s.
\end{equation}
Looking for $\tilde r_0$ in the form $\tilde r_0=e^{i\lambda x_1}G_\varphi^2 r_0$, we conclude from \eqref{eq_4_10} that 
\[
r_0=-e^{-i\lambda x_1}e^{sx_1}h^4\Delta^2 e^{-sx_1}v_s.
\]
Thus, it follows from \eqref{eq_4_1} that $\|r_0\|_{L^2(M)}=\mathcal{O}(h^{5/2})$, and therefore, using \eqref{eq_3_2}, we obtain that 
\begin{equation}
\label{eq_4_11}
\|\tilde r_0\|_{L^2(M)}=\mathcal{O}(h^{1/2}),
\end{equation}
as $h\to 0$.  Now $u_1$ given by \eqref{eq_4_7} is a solution to \eqref{eq_4_6_1} provided that 
\begin{equation}
\label{eq_4_12}
(P_\varphi^2+h^4q)e^{-i\lambda x_1}\tilde r_1=-h^4e^{\varphi/h}qu_0 \quad\text{in}\quad  M^{\text{int}}.
\end{equation}
Looking for $\tilde r_1$ in the form $\tilde r_1=e^{i\lambda x_1}G_\varphi^2 r_1$ with $r_1\in L^2(M)$, we see from \eqref{eq_4_12} that 
\begin{equation}
\label{eq_4_13}
(1+ h^4qG_\varphi^2)r_1=-h^4e^{\varphi/h}qu_0 \quad\text{in}\quad  M^{\text{int}}.
\end{equation}
In view of  \eqref{eq_3_2}, \eqref{eq_4_9},  \eqref{eq_4_1}, and \eqref{eq_4_11},  for all $0<h\ll 1$, there exists a unique solution $r_1\in L^2(M)$ to  \eqref{eq_4_13} such that 
\[
\|r_1\|_{L^2(M)}=\mathcal{O}(h^4)\|e^{\varphi/h}u_0\|_{L^2(M)}=\mathcal{O}(h^4), 
\]
and therefore, 
\begin{equation}
\label{eq_4_14}
\|\tilde r_1\|_{L^2(M)}=\mathcal{O}(h^2). 
\end{equation}

Next we would like to reconstruct the boundary traces $\gamma u_1=(u_1|_{\p M}, \p_\nu u_1|_{\p M})$, where the complex geometric optics solution $u_1$ to \eqref{eq_4_6_1} is given by \eqref{eq_4_7},  from the knowledge of the Dirichlet--to--Neumann map $\Lambda_q$.  First we claim that $u_1$ satisfies the equation 
\begin{equation}
\label{eq_4_15}
(1+ h^4e^{-\varphi/h}G_\varphi^2 qe^{\varphi/h})u_1=u_0.
\end{equation}
Indeed, applying the operator $G_\varphi^2$ to \eqref{eq_4_13} and then multiplying it by $e^{-\varphi/h}$, we get 
\begin{equation}
\label{eq_4_16}
e^{-sx_1}\tilde r_1+ h^4e^{-\varphi/h}G_\varphi^2 qe^{\varphi/h} u_1=0.
\end{equation}
Adding $u_0$ to both sides of \eqref{eq_4_16} gives us \eqref{eq_4_15}. 

Using Proposition \ref{prop_2_4}, we obtain from \eqref{eq_4_15} that $f=\gamma u_1\in H^{-1/2}(\p M)\times H^{-3/2}(\p M)$ satisfies the boundary integral equation  
\begin{equation}
\label{eq_4_17}
(1+ h^4S_\varphi(\Lambda_q-\Lambda_0))f=\gamma u_0. 
\end{equation}
Since $(M,g)$ is known, $u_0$ and therefore, $\gamma u_0$ are also known as well as the single layer operator $S_\varphi$, and the Dirichlet--to--Neumann map $\Lambda_0$. Furthermore, Dirichlet--to--Neumann map  $\Lambda_q$ is known as well.  By Proposition \ref{prop_2_5}, for all $0<h\ll 1$, the boundary trace $f=\gamma u_1$ can be reconstructed as the unique solution to \eqref{eq_4_17}. 

Now substituting $u_1$ and $u_2$, given by \eqref{eq_4_7} and \eqref{eq_4_4}, respectively,  into the integral identity \eqref{eq_2_26}, we get 
\begin{equation}
\label{eq_4_18}
\int_M q u_1\overline{u_2}dV=\langle (\Lambda_q-\Lambda_0)\gamma u_1, \gamma \overline{u_2} \rangle_{H^{1/2}(\p M)\times H^{3/2}(\p M), H^{-1/2}(\p M)\times H^{-3/2}(\p M)}.
\end{equation}
Now as $u_2$ solves $\Delta^2u_2=0$ in $M^{\text{int}}$, it is a known function. This together with the reconstruction of $\gamma u_1$ shows that the expression in the right hand side of \eqref{eq_4_18} can be reconstructed from our data. Thus, we can reconstruct the integral 
\begin{equation}
\label{eq_4_19}
\begin{aligned}
\int_M q u_1\overline{u_2}dV= \int_M qe^{-2i\lambda x_1}(\overline{w_s}v_s+\overline{\tilde r_2}(v_s+\tilde r_0+\tilde r_1)+\overline{w_s}(\tilde r_0+\tilde r_1))dV\\
=\int_M qe^{-2i\lambda x_1}\overline{w_s}v_sdV+ \mathcal{O}(h^{1/2}).
\end{aligned}
\end{equation}
Here we have used  \eqref{eq_4_7}, \eqref{eq_4_9}, \eqref{eq_4_4}, \eqref{eq_4_1}, \eqref{eq_4_2}, \eqref{eq_4_6}, \eqref{eq_4_11}, and \eqref{eq_4_14}. 

By Theorem \ref{thm_boundary}, we can determine $q|_{\p M}$ from the knowledge of $\Lambda_q$ and $(M,g)$ in a constructive way. Thus, we extend $q$ to a function in $C_0(\R\times M_0^{\text{int}})$ in such a way that $q|_{(\R\times M_0)\setminus M}$ is known. This together with \eqref{eq_4_19} and $dV=c^{\frac{n}{2}}dx_1dV_{g_0}$ allows us to reconstruct
\begin{equation}
\label{eq_4_20}
\int_\R e^{-2i\lambda x_1} \int_{M_0} q(x_1,x')\overline{w_s(x_1,x')}v_s(x_1,x')c(x_1,x')^{n/2}dV_{g_0}dx_1+ \mathcal{O}(h^{1/2}).
\end{equation}
Letting $h\to 0$ in \eqref{eq_4_20}, and using \eqref{eq_4_3}, we obtain from  \eqref{eq_4_20} that 
\begin{equation}
\label{eq_4_21}
\int_\R e^{-2i\lambda x_1}\int_0^L e^{-2\lambda t} q(x_1,\gamma(t))c(x_1,\gamma(t))dtdx_1
=\int_0^L \hat{\tilde q}(2\lambda, \gamma(t)) e^{-2\lambda t} dt,
\end{equation}
for any $\lambda\in \R$ and any non-tangential geodesic $\gamma$ in $M_0$. Here $\tilde q=qc$ and 
\[
\hat{\tilde q}(\lambda, x')=\int_\R e^{-i\lambda x_1}\tilde q(x_1,x')dx_1. 
\]
The integral in the right hand side of \eqref{eq_4_21} is the attenuated geodesic ray transform of $\hat{\tilde q}(2\lambda,\cdot)$ with constant attenuation $-2\lambda$. Note that if $M_0$ is simple then it was shown in \cite{Salo_Uhlmann_2011} that the attenuated ray transform is constructively invertible for any attenuation, and using the inversion procedure in \cite{Salo_Uhlmann_2011}, we reconstruct the potential $q$. 

In general, proceeding similarly to the end of the proof of \cite[Theorem 1.4]{Feizm_K_Oksan_Uhl}, using the constructive invertibility assumption of the geodesic ray transform on $M_0$, we reconstruct the potential $q$ in $M$. This completes the proof of Theorem \ref{thm_main}.

\begin{appendix}
\section{Boundary reconstruction of a continuous potential for the perturbed biharmonic operator}
\label{appendix}

The goal of this appendix is to give a reconstruction formula for the boundary values of a continuous potential $q$ from the knowledge of the Dirichlet--to--Neumann map for the perturbed biharmonic operator $\Delta^2+q$ on a smooth compact Riemannian manifold of dimension $n\ge 2$ with smooth boundary.  In the case of Schr\"odinger operator, the constructive determination of the boundary values of a continuous potential from boundary measurements is given in \cite[Appendix A]{Feizm_K_Oksan_Uhl}, and our reconstruction here will rely crucially on this work.  For the non-constructive boundary determination of a continuous potential in the case of the Schr\"odinger operator, we refer to the works \cite{Guillarmou_Tzou_2011}, \cite{Krup_Uhlmann_2020}, \cite{Ma_Tzou_2021}.  Our result is as follows. 
\begin{thm}
\label{thm_boundary}
Let $(M,g)$ be a given compact smooth Riemannian manifold of dimension $n\ge 2$ with smooth boundary, and let $q\in C(M)$ be such that assumption (A) is satisfied.  For each point $x_0\in \p M$, there exists an explicit family of functions $f_\lambda\in C^\infty(\p M)\times C^\infty(\p M)$, $0<\lambda\ll 1$, depending only on $(M,g)$, such that 
\[
q(x_0)=2\lim_{\lambda\to 0} \langle (\Lambda_q-\Lambda_0) f_\lambda,\overline{f_\lambda} \rangle_{H^{-3/2}(\p M)\times H^{-1/2}(\p M), H^{3/2}(\p M)\times H^{1/2}(\p M)}.
\]
\end{thm}

\begin{proof}
Let $f\in H^{3/2}(\p M)\times H^{1/2}(\p M)$ and let us start by considering the special case of the integral identity \eqref{eq_2_24},
\begin{equation}
\label{eq_app_2}
 \langle (\Lambda_q-\Lambda_0) f,\overline{f} \rangle_{H^{-3/2}(\p M)\times H^{-1/2}(\p M), H^{3/2}(\p M)\times H^{1/2}(\p M)}=\int_M qu \overline{v}dV.
\end{equation}
Here $u, v\in H^2(M^{\text{int}})$ are solutions to 
\begin{equation}
\label{eq_app_3}
\begin{cases}
(\Delta^2+q)u=0 & \text{in}\quad M^{\text{int}},\\
\gamma u=f & \text{on}\quad \p M,
\end{cases}
\end{equation}
and
\begin{equation}
\label{eq_app_4}
\begin{cases}
\Delta^2v=0 & \text{in}\quad M^{\text{int}},\\
\gamma v=f & \text{on}\quad \p M,
\end{cases}
\end{equation}
respectively.  

We would like to construct suitable solutions to \eqref{eq_app_3} and \eqref{eq_app_4} to test the integral identity \eqref{eq_app_2}. The construction of these solutions will be based on an explicit  family of functions $v_\lambda$, whose boundary values have a highly oscillatory behavior as $\lambda\to 0$, while becoming increasingly concentrated near a given point on the boundary of $M$. Such a family of functions $v_\lambda$  was introduced in  \cite{Brown_2001}, \cite{Brown_Salo_2006}, see also \cite{Feizm_K_Oksan_Uhl}, \cite{Krup_Uhlmann_2018}, \cite{Krup_Uhlmann_2020}, \cite{Krup_Uhlmann_2018_adv}. 

To define $v_\lambda$, we let $x_0\in \p M$ and let $(x_1,\dots, x_n)$ be the boundary normal coordinates centered at $x_0$ so that in these coordinates, $x_0 =0$, the boundary $\p M$ is given by $\{x_n=0\}$, and $M^{\text{int}}$ is given by $\{x_n > 0\}$.  In these local coordinates, we have $T_{x_0}\p M=\R^{n-1}$, equipped with the Euclidean metric.  The unit tangent vector $\tau$ is then given by $\tau=(\tau',0)$ where $\tau'\in \R^{n-1}$, $|\tau'|=1$.   Associated to the tangent vector $\tau'$ is the covector $\xi'_\alpha=\sum_{\beta=1}^{n-1} g_{\alpha \beta}(0) \tau'_\beta=\tau'_\alpha\in T^*_{x_0}\p M$. 

Let $\eta\in C^\infty_0(\R^n;\R)$ be such that $\supp(\eta)$ is in a small neighborhood of $0$, and 
\begin{equation}
\label{eq_app_5}
\int_{\R^{n-1}}\eta(x',0)^2dx'=1.
\end{equation}
Let $\frac{1}{3}\le \alpha\le \frac{1}{2}$. Following \cite{Brown_Salo_2006}, \cite[Appendix C]{Krup_Uhlmann_2020}, \cite[Appendix A]{Feizm_K_Oksan_Uhl} in the boundary normal coordinates,  we set 
\begin{equation}
\label{eq_app_6}
v_\lambda(x)= \lambda^{-\frac{\alpha(n-1)}{2}-\frac{1}{2}}\eta\bigg(\frac{x}{\lambda^{\alpha}}\bigg)e^{\frac{i}{\lambda}(\tau'\cdot x'+ ix_n)}, \quad 0<\lambda\ll 1,
\end{equation}
so that  $v_\lambda\in C^\infty(M)$, with $\supp(v_\lambda)$ in  $\mathcal{O}(\lambda^{\alpha})$ neighborhood of $x_0=0$. Here $\tau'$ is viewed as a covector.  A direct computation shows that 
\begin{equation}
\label{eq_app_7}
\|v_\lambda\|_{L^2(M)}=\mathcal{O}(1),  
\end{equation}
as $\lambda\to 0$, see also \cite[Appendix C]{Krup_Uhlmann_2020}.  Following \cite[Appendix A]{Feizm_K_Oksan_Uhl}, we let 
\begin{equation}
\label{eq_app_8}
v=v_\lambda+r_1,
\end{equation} 
where $r_1\in H^1_0(M^{\text{int}})$ is the solution to the Dirichlet problem, 
\begin{equation}
\label{eq_app_9}
\begin{cases}
-\Delta r_1=\Delta v_\lambda & \text{in}\quad M^{\text{int}},\\
r_1|_{\p M}=0.
\end{cases}
\end{equation} 
By boundary elliptic regularity, we have $r_1\in C^\infty(M)$, and therefore, $v\in C^\infty(M)$. 
It was established in \cite[Appendix A]{Feizm_K_Oksan_Uhl} that when $\alpha=1/3$, 
\begin{equation}
\label{eq_app_10}
\|r_1\|_{L^2(M)}=\mathcal{O}(\lambda^{1/12}),
\end{equation}
as $\lambda\to 0$. In what follows,  we fix $\alpha=1/3$. 

Note that $v\in C^\infty(M)$ solves the Dirichlet problem \eqref{eq_app_4} with 
\begin{equation}
\label{eq_app_11}
f=f_\lambda:=(v_\lambda|_{\p M}, \p_\nu(v_\lambda+r_1)|_{\p M}).
\end{equation}
Now since the manifold $(M,g)$ is known, the harmonic function $v$ as well as the trace $f_\lambda$ are known. 

Next we look for a solution $u$ to \eqref{eq_app_3} with the Dirichlet data $f=f_\lambda$ given by \eqref{eq_app_11} in the form
\begin{equation}
\label{eq_app_12}
u=v_\lambda+r_1+r_2.
\end{equation}
Thus, $r_2\in H^2(M^{\text{int}})$ is the solution to the following Dirichlet problem,
\begin{equation}
\label{eq_app_13}
\begin{cases}
(\Delta^2+q)r_2=-q(v_\lambda+r_1) & \text{in}\quad M^{\text{int}},\\
\gamma r_2=0 & \text{on}\quad \p M.
\end{cases}
\end{equation}
It follows from \cite[Section 11, p. 325, 326]{Grubb_book} that for all $s>3/2$, 
\begin{equation}
\label{eq_app_14}
\|r_2\|_{H^s(M^{\text{int}})}\le C\|q(v_\lambda+r_1)\|_{H^{s-4}(M^{\text{int}})}.
\end{equation}
In particular, letting $s=3$ in \eqref{eq_app_14}, we get 
\begin{equation}
\label{eq_app_15}
\begin{aligned}
\|r_2\|_{L^2(M)}\le C\|q(v_\lambda+r_1)\|_{H^{-1}(M^{\text{int}})}&\le C(\|q v_\lambda\|_{H^{-1}(M^{\text{int}})}+ \|r_1\|_{L^2(M)})\\
&=o(1)+\mathcal{O}(\lambda^{1/12})=o(1),
\end{aligned}
\end{equation}
as $\lambda\to 0$. Note that here we used the following bound 
\[
\|q v_\lambda\|_{H^{-1}(M^{\text{int}})}=o(1),
\]
as $\lambda\to 0$, cf. \cite[Appendix A, (A.20)]{Feizm_K_Oksan_Uhl},  
together with \eqref{eq_app_10}. 

Substituting $v$ and $u$ given by \eqref{eq_app_8} and \eqref{eq_app_12}, respectively, into \eqref{eq_app_2} and taking the limit $\lambda\to 0$, we obtain that 
\begin{equation}
\label{eq_app_16}
\lim_{\lambda\to 0} \langle (\Lambda_q-\Lambda_0) f_\lambda,\overline{f_\lambda} \rangle_{H^{-3/2}(\p M)\times H^{-1/2}(\p M), H^{3/2}(\p M)\times H^{1/2}(\p M)}=\lim_{\lambda\to 0} (I_1+I_2),
\end{equation}
where 
\[
I_1=\int_{M} q|v_\lambda|^2dV,\quad I_2=\int_{M} q(v_\lambda \overline{r_1}+(r_1+r_2)(\overline{v_\lambda}+\overline{r_1}))dV.
\]
Using \eqref{eq_app_10} and \eqref{eq_app_15}, we get 
\begin{equation}
\label{eq_app_17}
\lim_{\lambda\to 0}I_2=0.
\end{equation}
A direct computation shows that 
\begin{equation}
\label{eq_app_18}
\lim_{\lambda\to 0}I_1=\frac{1}{2}q(0),
\end{equation}
cf. \cite[Appendix A, (A.24)]{Feizm_K_Oksan_Uhl}. Combining \eqref{eq_app_16}, \eqref{eq_app_17}, and \eqref{eq_app_18}, we see that 
\[
q(0)=2\lim_{\lambda\to 0} \langle (\Lambda_q-\Lambda_0) f_\lambda,\overline{f_\lambda} \rangle_{H^{-3/2}(\p M)\times H^{-1/2}(\p M), H^{3/2}(\p M)\times H^{1/2}(\p M)}.
\]
This completes the proof of Theorem \ref{thm_boundary}. 
\end{proof}

\end{appendix}

\section*{Acknowledgements}

The author wishes to express her sincere thanks to Mikhail Belishev and Roman Novikov for the stimulating questions during her talk at the conference Days on Diffraction 2021, that motivated this paper. The author gratefully acknowledges the many helpful suggestions of Katya Krupchyk during the preparation of this paper. The author is partially supported by the National Science Foundation  (DMS 1815922 and DMS 2109199).

\end{document}